\documentclass{amsart}

\usepackage[cmtip,all]{xy}
\usepackage{enumerate}
\usepackage{hyperref}

\DeclareMathOperator{\Gr}{Gr}
\DeclareMathOperator{\pr}{pr}
\DeclareMathOperator{\ind}{ind}
\DeclareMathOperator{\Sq}{Sq}
\DeclareMathOperator{\gr}{gr}
\DeclareMathOperator{\rank}{rank}

\DeclareMathOperator{\Spin}{Spin}
\DeclareMathOperator{\SU}{SU}
\DeclareMathOperator{\SO}{SO}
\DeclareMathOperator{\Sp}{Sp}
\DeclareMathOperator{\GL}{GL}
\DeclareMathOperator{\Cl}{Cl}
\DeclareMathOperator{\Diff}{Diff}

\DeclareMathOperator{\BSO}{BSO}
\DeclareMathOperator{\BSpin}{BSpin}

\newcommand\RP{\mathbb R\rm P}

\newcommand\goe{\mathfrak g}
\newcommand\aoe{\mathfrak a}
\newcommand\hoe{\mathfrak h}

\newcommand\Z{\mathbb Z}
\newcommand\R{\mathbb R}
\newcommand\C{\mathbb C}

\newcommand\ii{\mathbf i}
\newcommand\jj{\mathbf j}
\newcommand\kk{\mathbf k}
\newcommand\II{\mathbf I}
\newcommand\JJ{\mathbf J}
\newcommand\KK{\mathbf K}

\newcommand\hol{{\rm hol}}

\theoremstyle{plain}
  \newtheorem{theorem}{Theorem}
  \newtheorem{lemma}{Lemma}
  \newtheorem{corollary}{Corollary}
\theoremstyle{definition}
  \newtheorem{example}{Example}
\theoremstyle{remark}
  \newtheorem*{remark*}{Remark}

\numberwithin{equation}{section}

\begin{document}

\title[Rank two distributions of Cartan type]{On 5-manifolds admitting rank two distributions of Cartan type}

\author{Shantanu Dave}

\address{Shantanu Dave,
         Wolfgang Pauli Institute
         c/o Faculty of Mathematics,
         University of Vienna,
         Oskar-Morgenstern-Platz 1,
         1090 Vienna,
         Austria.}

\email{shantanu.dave@univie.ac.at}

\thanks{S.~D.\ was supported by the Austrian Sciences Fund (FWF) grant P24420.}

\author{Stefan Haller}

\address{Stefan Haller, 
         Max Planck Institute for Mathematics, 
         Vivatsgasse 7,
         53111 Bonn, 
         Germany.}
         
\curraddr{Department of Mathematics,
          University of Vienna,
          Oskar-Morgenstern-Platz 1,
          1090 Vienna,
          Austria.}

\email{stefan.haller@univie.ac.at}

\thanks{The second author would like to express his gratitude to the Max Planck Institute for Mathematics in Bonn for the hospitality during an extended visit when substantial parts of this work were done.}



\begin{abstract}
We consider the question whether an orientable $5$-manifold can be equipped with a rank two distribution of Cartan type and what $2$-plane bundles can be realized.
We obtain a complete answer for open manifolds.
In the closed case, we settle the topological part of this problem and present partial results concerning its geometric aspects and new examples.
\end{abstract}

\maketitle


\section{Introduction}\label{S:intro}

Let M be a smooth $5$-manifold, and suppose $\xi\subseteq TM$ is a rank two distribution, i.e., a smooth subbundle of rank two. 
Such a distribution is said to be of Cartan type \cite{BH93} iff it is bracket generating with growth vector $(2,3,5)$, that is, 
iff Lie brackets of vector fields tangential to $\xi$ span a rank three subbundle of $TM$, and triple brackets of vector fields tangential to $\xi$ span all of the tangent bundle. 
These are also known as generic rank two distributions in dimension five \cite{S06a,S08,CS09a,HS11}.

Generic rank two distributions in dimension five have a long history reaching back to Cartan's ``five variables paper'' \cite{C10}. 
The mechanical system of a surface rolling without slipping and twisting on another surface gives rise to a $5$-dimensional configuration
space equipped with a rank two distribution encoding the no slipping and twisting condition. 
This provides a basic example of Cartan type, provided the values of the Gaussian curvatures are disjoint, see \cite[Section 4.4]{BH93}. 
If both surfaces are round spheres, and the ratio of their radii is $1:3$, then the universal (double) covering of this configuration space is isomorphic to the ``flat'' model, a homogeneous  space for the real split form of the exceptional Lie group $G_2$, see \cite{C10,BM05}.
In fact, the geometry of these distributions can equivalently be described as a parabolic geometry of type $(G_2,P)$, where 
$P$ is a particular parabolic subgroup of the exceptional Lie group $G_2$, see \cite{C10,CS09,S08}. 
The most symmetric example, the flat model, is the homogeneous space $G_2/P$ with underlying manifold diffeomorphic to $S^2\times S^3$.

In this paper we will address the existence problem: 
Which $5$-manifolds admit a rank two distribution of Cartan type, and which $2$-plane bundles $\xi$ can be realized? 
Although Cartan type rank two distributions in five dimensions have been extensively studied and are well known for their rich geometry, a systematic approach to their existence seems unavailable in the literature.

Let us recall some of these fascinating geometrical facts.
More than a century ago, Cartan \cite{C10} constructed what is now referred to as the canonical Cartan connection associated with a generic rank two distribution in dimension five that leads to some striking consequences:
The symmetry group of such a distribution is a Lie group of dimension at most $14$ and is subject to further restrictions, see \cite[Theorem~7]{CN09}.
Cartan's harmonic curvature, a section of $S^4\xi^*$, is the complete obstruction to local flatness, that is, the harmonic curvature vanishes if and only if the distribution is locally isomorphic to the flat model $G_2/P$.
Nurowski \cite{N05} constructed a canonical conformal metric of signature $(2,3)$ whose conformal holonomy is contained in $G_2$, see also \cite{S08,CS09a}.
Parabolic geometry also permits to construct differential operators naturally associated with a generic rank two distribution in dimension five and representations of $G_2$, see \cite{CSS01,CS12}.

The approach presented here is to split the problem of existence of Cartan distributions into two steps. 
First we observe that there are immediate topological obstructions.
More precisely, if $\xi$ is an orientable $2$-plane field of Cartan type on $M$, then there exists an isomorphism of vector bundles $TM\cong\xi\oplus\varepsilon^1\oplus\xi$.
In Section~\ref{S:top} we characterize the $5$-manifolds $M$ and $2$-plane bundles $\xi$ which admit such an isomorphism, see Theorem~\ref{T:top}.
It is easy to see that these $5$-manifolds must always be spinnable and their characterization in the open case essentially follows from the fact that the fractional Pontryagin class induces a $5$-equivalence, $\frac{1}{2}p_1\colon\BSpin(5)\to K(\Z,4)$.
In the closed case, the Kervaire semi-characteristic has to vanish according to a result of Atiyah \cite{A70}.
To achieve a corresponding characterization for closed manifolds, we build on K-theoretic results of Atiyah--Dupont \cite{AD72}, a recent generalization due to Tang--Zhang \cite{TZ02}, and verify that the mod 2 index of a certain Dirac operator vanishes.

On open manifolds, if the topological requirements are met then there are no further obstructions to the existence of rank two distributions of Cartan type.
Indeed, being of Cartan type can be considered as an open and $\Diff(M)$ invariant differential relation, hence Gromov's h-principle for open manifolds is applicable.
Eventually, this permits us to obtain a general existence result for open $5$-manifolds, see Theorem~\ref{T:open} in Section~\ref{S:h}.

However, we do not know to what extent rank two distributions of Cartan type on closed $5$-manifolds satisfy the h-principle. 
In particular, it is unclear if every closed $5$-manifold whose tangent bundle can be decomposed as above admits a $2$-plane bundle of Cartan type.
In Section~\ref{S:princ}, using the Lutz--Martinez theorem, we construct some new examples on closed manifolds of the form $\Sigma\times N$ where $\Sigma$ is a surface and $N$ is an orientable $3$-manifold.
If further symmetries are imposed, the h-principle might fail.
For instance, if $\Sigma$ is a closed connected surface of Euler characteristic $\chi(\Sigma)\leq0$, then $\Sigma\times T^3$ does not admit a principal $T^3$-connection whose $2$-plane bundle is of Cartan type, see Lemma~\ref{L:aoe}(a). 
Similar phenomena can be found in Theorem~\ref{T:extension} and Corollary~\ref{C:ext}.

To put these results in perspective, recall that the Lutz--Martinez theorem, see \cite[Theorem 4.3.1]{G08}, asserts that every
cooriented tangent $2$-plane field on a closed oriented 3-manifold is homotopic to a contact structure,
and a result of Geiges \cite{G91} states that a closed simply connected $5$-manifold admits a contact structure in every homotopy class of almost contact structures. 
Generalizing a classification due to Eliashberg \cite{E89}, it has recently been shown \cite{BEM15} that there is a higher dimensional analog of overtwisted contact structures which abide by an h-principle.
Thus, in all dimensions, every almost contact structure is homotopic to an (overtwisted) contact structure.
On the other hand, existence of symplectic structures on closed manifolds is a subtle issue, see \cite[Chapter 11]{EM02}.

Let us also point out a result of Vogel, see \cite{V09}, which asserts that every parallelizable $4$-manifold admits an (orientable) Engel structure.
Recall that an Engel structure on a smooth $4$-manifold is a rank two distribution with growth vector $(2,3,4)$.
Like contact structures, Engel structures admit local normal forms and thus do not have local geometry.

It is conceivable that the canonical Cartan connection and the naturally associated differential operators will help clarify the questions of whether the h-principle holds true for rank two distributions of Cartan type on closed $5$-manifolds.
This appears to be an intriguing open problem.
In a recent preprint, see \cite{DH17}, it has been shown that the BGG sequences associated with a rank two distribution of Cartan type are all hypoelliptic.

We would both like to thank Andreas \v Cap for encouraging conversations.
We would also like to express our gratitude to two anonymous referees for helpful comments and recommendations which, in our opinion, improved the quality of the text considerably.

\section{Almost generic rank two distributions in dimension five}\label{S:top}

Let $M$ be a closed $5$-manifold and suppose $\xi\subseteq TM$ is a rank two distribution of Cartan type.
Hence, Lie brackets of sections of $\xi$ span a rank three subbundle we will denote by $\eta:=[\xi,\xi]\subseteq TM$.
The Lie bracket turns the associated graded bundle, $\gr(TM):=\xi\oplus(\eta/\xi)\oplus(TM/\eta)$, into a bundle of graded nilpotent Lie algebras.
Since $\xi$ is of Cartan type, the non-trivial components of this algebraic (Levi) bracket are bundle isomorphisms: $\Lambda^2\xi\cong\eta/\xi$ and $\xi\otimes(\eta/\xi)\cong TM/\eta$.
Since $TM\cong\gr(TM)$, we obtain an isomorphism of vector bundles,
\begin{equation}\label{E:TMdecom} 
TM\cong\xi\oplus\Lambda^2\xi\oplus(\xi\otimes\Lambda^2\xi).
\end{equation}
Note that $\gr(TM)$ is a locally trivial bundle of nilpotent Lie algebras, the fibers are all isomorphic to the $5$-dimensional nilpotent Lie algebra with basis: $X$, $Y$, $[X,Y]$, $[X,[X,Y]]$, $[Y,[X,Y]]$.

Given an isomorphism as in \eqref{E:TMdecom}, an orientation of $\xi$ determines an orientation of $M$ and vice versa.
In particular, $M$ is orientable iff $\xi$ is. 
In the orientable case $\Lambda^2\xi\cong\varepsilon^1$, hence there exists an isomorphism as in \eqref{E:TMdecom} if and only if
\begin{equation}\label{E:TMdecoms}
TM\cong\xi\oplus\varepsilon^1\oplus\xi,
\end{equation}
where $\varepsilon^1$ denotes the trivial line bundle.

The purpose of this section is to establish the subsequent theorem, providing
necessary and sufficient conditions for the existence of an isomorphism as in \eqref{E:TMdecoms}.

A manifold $M$ will be called spinnable if it admits a spin structure, equivalently, if $w_1(M)=0=w_2(M)$, where $w_i(M)$ denote the Stiefel--Whitney classes.
For a closed connected manifold $M$, we let $k(M)\in\Z_2$ denote its (real) Kervaire semi-characteristic \cite{K56},
$$
k(M)=\sum_{\text{$q$ even}}\dim H^q(M;\R)\mod2.
$$
In view of the Lusztig--Milnor--Peterson formula, see \cite{LMP69}, this coincides with the $\Z_2$ semi-characteristic $k_2(M)=\sum_{\text{$q$ even}}\dim H_q(M;\Z_2)$ for spinnable $M$.
Recall that a manifold is called open if it has no closed connected component.

\begin{theorem}\label{T:top}
a) Let $\xi$ be an orientable $2$-plane bundle over an open $5$-manifold $M$.
Then $TM\cong\xi\oplus\varepsilon^1\oplus\xi$ if and only if $M$ is spinnable and $\tfrac12p_1(M)=e(\xi)^2$.

b) Let $\xi$ be an orientable $2$-plane bundle over a closed connected $5$-manifold $M$.
Then $TM\cong\xi\oplus\varepsilon^1\oplus\xi$ iff $M$ is spinnable, $\tfrac12p_1(M)=e(\xi)^2$, and $k(M)=0$.
\end{theorem}

The first part of this theorem will be deduced using elementary homotopy theory,
$\tfrac12p_1(M)\in H^4(M;\Z)$ denotes the first fractional Pontryagin class to be recalled below.
The second part of Theorem~\ref{T:top} is more subtle, the proof we will give rests on results of Thomas \cite{T68},
Atiyah \cite{A70}, Atiyah--Dupont \cite{AD72} and a recent generalization of the latter due to Tang and Zhang \cite{TZ02}.

Let us point out that the Stiefel--Whitney classes of a closed spinnable $5$-manifold are all trivial according to Wu's formula, see \cite[Theorem~11.14]{MS}.
Moreover, since the spin cobordism group $\Omega^{\rm spin}_5$ is trivial, see \cite[page~201]{M63} or \cite{ABP66,ABP67}, every closed spinnable $5$-manifold bounds a compact spinnable $6$-manifold.
Let us also point out that Rokhlin's theorem imposes a restriction on the first Pontryagin class of spinnable $5$-manifolds. 
More precisely, if $M$ is a closed spinnable $5$-manifold, then
\begin{equation}\label{E:rokhlin}
\langle p_1(M),a\rangle\equiv0\mod48,
\end{equation}
for every integral homology class $a\in H_4(M;\Z)$.
Indeed, by Poincar\'e duality, and since $H^1(M;\Z)=[M,S^1]$, every homology class in $H_4(M;\Z)$ can be represented by a closed 4-dimensional submanifold $N$ with trivial normal bundle. 
These submanifolds are spinnable, and $\hat A(N)=-\frac1{24}\langle p_1(N),[N]\rangle=-\frac1{24}\langle p_1(M),[N]\rangle$ is an even integer by Rokhlin's theorem, see \cite[Theorem~IV.1.1]{LM89}.

Using Smale's classification \cite{S62} of closed simply connected spinnable $5$-ma\-ni\-folds, see also \cite{B65} or \cite[Section~2.1]{C11}, we obtain the following explicit description of all simply connected manifolds satisfying the condition in Theorem~\ref{T:top}(b):

\begin{remark*}
For a closed simply connected $5$-manifold $M$ the following are equivalent:
\begin{enumerate}[(a)]
\item 
$TM\cong\xi\oplus\varepsilon^1\oplus\xi$ for one (and then every) $2$-plane bundle $\xi$ over $M$.
\item 
$w_2(M)=0$ and $\rank H_2(M;\Z)$ is odd.
\item 
$M$ is parallelizable.
\item 
There exists a diffeomorphism
$$
M\cong\underbrace{(S^2\times S^3)\sharp\cdots\sharp(S^2\times S^3)}_\text{odd number of summands}\sharp M_{k_1}\sharp\cdots\sharp M_{k_n}
$$ 
where $M_k$ denotes the unique (up to diffeomorphism) closed simply connected
spinnable $5$-manifold with $H_2(M_k;\Z)\cong\Z_k\oplus\Z_k$.
\item 
$M\cong\partial W$ where $W\subseteq\R^6$ is a compact $6$-manifold with boundary which has the homotopy type of a $3$-dimensional simply connected CW complex such that the rank of $H_2(W;\Z)$ is odd and $H_3(W;\Z)=0$.
\end{enumerate}
In this case $\xi\oplus\xi=\varepsilon^4$ for every $2$-plane bundle $\xi$.
Indeed, since $H^3(M,W;\Z)\cong H_3(W;\Z)=0$, the inclusion induces a surjection $H^2(W;\Z)\to H^2(M;\Z)$, hence $\xi$ is the restriction of a $2$-plane bundle $\tilde\xi$ over $W$, and the bundle $\tilde\xi\oplus\tilde\xi$ is trivial since it admits a quaternionic structure and $W$ is
homotopically $3$-dimensional.
\end{remark*}

To discuss the first fractional Pontryagin class, suppose $\sigma$ is a spinnable vector bundle over $X$, i.e.\ $w_1(\sigma)=0=w_2(\sigma)$. 
Then the mod $2$ reduction of its first Pontryagin class vanishes, for it coincides with $w_2(\sigma)^2\in H^4(X;\Z_2)$, see \cite[Problem~15-A]{MS}.
Since $H^4(\BSO(k);\Z)$ is torsionfree,\footnote{For $4<k$, the group $H^4(\BSO(k);\Z)=H^4(\BSO;\Z)\cong\Z$ is generated by $p_1$. 
Moreover, $H^4(\BSO(4);\Z)\cong\Z^2$ with generators $p_1$ and $e$, and $H^4(\BSO(3);\Z)\cong\Z\cong H^4(\BSO(2);\Z)$.} these bundles have a well defined fractional first Pontryagin class, $\tfrac12p_1(\sigma)\in H^4(X;\Z)$, which is natural and stable. The corresponding classifying map, 
\begin{equation}\label{E:12p1}
\tfrac12p_1\colon\BSpin\to K(\Z,4),
\end{equation} 
is an $8$-equivalence. Consequently, two spinnable bundles over a 7-dimensional CW complex are stably isomorphic 
iff they have the same first fractional Pontryagin class, see \cite{DW59}.
We will use the following simple fact repeatedly.

\begin{lemma}\label{L:12p1xixi}
If $\xi$ is an orientable $2$-plane bundle, then $\xi\oplus\xi$ is spinnable and 
$$
\tfrac12p_1(\xi\oplus\xi)=e(\xi)^2.
$$
\end{lemma}

\begin{proof}
Since $\xi$ is orientable it admits a complex structure. The structure group of $\xi\oplus\bar\xi$ can be reduced to $\SU(2)$, hence $\xi\oplus\xi$ is spinnable.
Moreover, $c_1(\xi\oplus\bar\xi)=0$ and $p_1(\xi\oplus\xi)=-2c_2(\xi\oplus\bar\xi)=2c_1(\xi)^2=2e(\xi)^2$, see \cite[Corollary~15.5]{MS}.
Using the universal bundle over $\BSO(2)$ and the fact that $H^4(\BSO(2);\Z)\cong\Z$ does not have $2$-torsion, we conclude $\tfrac12p_1(\xi\oplus\xi)=e(\xi)^2$.
\end{proof}

We are now in a position to give a proof of Theorem~\ref{T:top}(a).
One implication follows immediately from Lemma~\ref{L:12p1xixi}. For the converse suppose $M$ is spinnable and assume $\tfrac12p_1(M)=e(\xi)^2$.
By Lemma~\ref{L:12p1xixi} the bundle $\xi\oplus\varepsilon^1\oplus\xi$ is spinnable, and 
\begin{equation}\label{E:abc}
\tfrac12p_1(TM)=\tfrac12p_1(\xi\oplus\varepsilon^1\oplus\xi).
\end{equation}
Recall that an open $5$-manifold has the homotopy type of a $4$-dimensional CW complex.
Since the map $\frac12p_1\colon\BSpin(5)\to K(\Z,4)$ is a $5$-equivalence,\footnote{Recall that $\Spin(5)\cong\Sp(2)$ fibers over $S^7$ with typical fiber $\Sp(1)\cong S^3$.} it induces a bijection $[M,\BSpin(5)]\cong[M,K(\Z,4)]=H^4(M;\Z)$.
Combining this with \eqref{E:abc}, we obtain $TM\cong\xi\oplus\varepsilon^1\oplus\xi$, hence the first part of Theorem~\ref{T:top}.

Let us now turn to the proof of the second part of Theorem~\ref{T:top}. One implication follows from Lemma~\ref{L:12p1xixi} and the subsequent special case of a more general result of Atiyah's in $4k+1$ dimensions \cite[Theorem~4.1]{A70}:

\begin{lemma}[Atiyah]\label{L:atiyah}
If an orientable closed $5$-manifold $M$ admits an orientable field of tangent $2$-planes, then $k(M)=0$.
\end{lemma}

To show the converse implication in Theorem~\ref{T:top}(b), suppose $M$ is spinnable, $\tfrac12p_1(M)=e(\xi)^2$, and $k(M)=0$.
Note that the bundles $TM$ and $\xi\oplus\varepsilon^1\oplus\xi$ are stably isomorphic. 
Indeed, they are both spinnable with the same fractional Pontryagin class, see Lemma~\ref{L:12p1xixi}, and the map \eqref{E:12p1} is an $8$-equivalence.
A result of Thomas, see \cite[Lemma~3]{T68}, permits to conclude that these two bundles are indeed isomorphic, provided both admit two linearly independent sections.
To apply Thomas' lemma, one also has to check $\Sq^2(u)+w_2(M)u=0$ for all $u\in H^3(M;\Z_2)$, where $\Sq^2$ denotes the second Steenrod square. 
The latter follows from Wu's formula and $w_1(M)=0$.

The tangent bundle does indeed admit two linearly independent sections in view of the following special case of a more general statement for $(4k+1)$-dimensional manifolds due to Atiyah and Dupont \cite{AD72}, see also \cite[Table~2]{T69}.

\begin{lemma}[Atiyah--Dupont]\label{L:atiyah-dupont}
An orientable closed connected $5$-manifold $M$ admits two linearly independent vector fields if and only if $w_4(M)=0$ and $k(M)=0$.
\end{lemma}

Alternatively, one could combine a result of Thomas', see \cite[Corollary~1.2]{T67} or \cite[Theorem~3(a)]{T68}, with the Lusztig--Milnor--Peterson formula \cite{LMP69} 
to obtain the following weaker statement which would be sufficient for our purpose, see \cite[Theorem~3]{T69}: A closed connected spinnable $5$-manifold $M$ admits two linearly independent vector fields if and only if $k(M)=0$.

To complete the proof of Theorem~\ref{T:top}, it therefore suffices to establish Lemma~\ref{L:spanxixi} below.
We have not been able to find this statement in the literature. 
However, if, in addition, one assumes $w_2(\xi)=0$, then the result follows from Thomas' work, cf.\ the proof of \cite[Theorem~4.1]{T67}.

\begin{lemma}\label{L:spanxixi}
If $\xi$ is an orientable $2$-plane bundle over a closed spinnable $5$-manifold, then $\xi\oplus\varepsilon^1\oplus\xi$ admits two linearly independent sections.
\end{lemma}

The proof of this lemma is based on the following generalization of Lemma~\ref{L:atiyah-dupont} due to Tang and Zhang, see \cite[Corollary~3.7]{TZ02}.

\begin{lemma}[Tang--Zhang]\label{L:tang-zhang}
Let $E$ be an orientable $5$-plane bundle over a closed connected orientable $5$-manifold $M$ such that $w_2(E)=w_2(M)$.
Then $E$ admits two linearly independent sections if and only if $w_4(E)=0$ and $\ind_2(P_E)=0$.
\end{lemma}

Here $\ind_2(P_E)$ denotes the $\Z_2$-index of a certain Dirac operator $P_E$ constructed using the bundle $E$.
Before discussing this operator in detail, let us verify that Lemma~\ref{L:tang-zhang} is applicable to $E=\xi\oplus\varepsilon^1\oplus\xi$ in the situation of Lemma~\ref{L:spanxixi}.
Clearly, $w_2(\xi\oplus\varepsilon^1\oplus\xi)=0=w_2(M)$ in view of Whitney's formula and the spinnability of $M$.
Moreover, $w_4(\xi\oplus\varepsilon^1\oplus\xi)=w_2(\xi)^2=\Sq^2(w_2(\xi))=0$, since the second Steenrod square on a closed spinnable $5$-manifold vanishes completely.
Indeed, by Wu's formula, all Wu classes of $M$ are trivial, hence the Steenrod squares $\Sq^2\colon H^3(M;\Z_2)\to H^5(M;\Z_2)$ and $\Sq^1\colon H^4(M;\Z_2)\to H^5(M;\Z_2)$ both vanish,
and thus the Cartan formula \cite[\S8]{MS} and the Adem relation $\Sq^1\Sq^1=0$ give $0=\Sq^2(xy)=\Sq^2(x)y+\Sq^1(x)\Sq^1(y)=\Sq^2(x)y+\Sq^1(\Sq^1(x)y)=\Sq^2(x)y$,
for all $x\in H^2(M;\Z_2)$ and $y\in H^1(M;\Z_2)$. By Poincar\'e duality, this implies that the (Steenrod) square $\Sq^2\colon H^2(M;\Z_2)\to H^4(M;\Z_2)$ vanishes too.
To prove Lemma~\ref{L:spanxixi} it thus remains to show $\ind_2(P_E)=0$ for the bundle $E=\xi\oplus\varepsilon^1\oplus\xi$.

Let us now describe Tang and Zhang's operator $P_E$ in detail and prove a corresponding vanishing result for its $\Z_2$-index.
Let $X$ be an oriented Riemannian manifold of dimension $4k+1$ and suppose $E$ is an oriented Euclidean vector bundle of rank $4k+1$ over $X$, such that $w_2(E)=w_2(X)$.
Then $TX\oplus E$ is a spinnable vector bundle of rank $8k+2$.
We equip $TX\oplus E$ with the fiberwise Euclidean inner product obtained from the Riemannian metric on $TX$ and the fiberwise Euclidean metric on $E$.
Furthermore, we equip $TX$ with the Levi--Civita connection, we choose a linear connection on $E$ such that the fiberwise Euclidean metric is parallel, and we equip $TX\oplus E$ with the direct sum connection.
In particular, $TX$ and $E$ are orthogonal, and the metric on $TX\oplus E$ is parallel.

Fix a spin structure $\tilde F\to F$ for $TX\oplus E$.
Here $F$ denotes the orthonormal frame bundle of $TX\oplus E$, $\tilde F$ is a principal $\Spin(8k+2)$ bundle over $M$, and $\tilde F\to F$ is a bundle map which is equivariant over the homomorphism of structure groups $\Spin(8k+2)\to\SO(8k+2)$.
Let $S$ denote the irreducible module of the Clifford algebra $\Cl(\R^{8k+2})$, and let $S(TX\oplus E):=\tilde F\times_{\Spin(8k+2)}S$ denote the associated spinor bundle.
Each fiber of $S(TX\oplus E)$ is an irreducible module of the corresponding fiber in the Clifford bundle $\Cl(TX\oplus E):=F\times_{\SO(8k+2)}\Cl(\R^{8k+2})$.
We will denote this Clifford multiplication by $c$.
The linear connection on $TX\oplus E$ provides a principal connection on $F$ which lifts to a principal connection on $\tilde F$.
The latter induce linear connections on $\Cl(TX\oplus E)$ and $S(TX\oplus E)$ such that Clifford multiplication is parallel, $\nabla c=0$, see \cite[Proposition~II.4.11]{LM89}.
Since $\rank(TX\oplus E)\equiv2\mod8$, the module $S$ admits a quaternionic structure which commutes with Clifford multiplication, see \cite[Table~2]{ABS64}.
This induces a quaternionic structure on the spinor bundle which commutes with Clifford multiplication by sections of $\Cl(TX\oplus E)$.
We will denote multiplication with the standard unit quaternions on $S(TX\oplus E)$ by $\ii$, $\jj$, and $\kk$, respectively.
Clearly, the quaternionic structure on the spinor bundle is parallel, i.e., $\nabla\ii=\nabla\jj=\nabla\kk=0$.
Recall that the module $S$ admits a Euclidean inner product such that Clifford multiplication with unit vectors from $\R^{8k+2}\subseteq\Cl(\R^{8k+2})$, as well as multiplication with unit quaternions, are all orthogonal, see \cite[Proposition~I.5.16]{LM89}.
We equip $S(TX\oplus E)$ with the induced fiberwise Euclidean metric.
Clearly, this metric is parallel.
Moreover, $\ii^*=-\ii$, $\jj^*=-\jj$, $\kk^*=-\kk$, $c(Y)^*=-c(Y)$, and $c(e)^*=-c(e)$ for every vector field $Y$ and each section $e$ of $E$.
We denote the corresponding Atiyah--Singer--Dirac operator acting on sections of $S(TX\oplus E)$ by
$$
D_E:=\sum_{i=1}^{4k+1}c(Y_i)\nabla_{Y_i}.
$$
Here $Y_1,\dotsc,Y_{4k+1}$ is a local orthonormal frame of $TX$.
Clearly, $D_E$ commutes with $\ii$, $\jj$, and $\kk$.
Moreover, $D_E$ is formally self-adjoint, see \cite[Proposition~II.5.3]{LM89}.

Deviating slightly from the notation used in \cite{TZ02}, we let $\JJ=c(Y_1)\cdots c(Y_{4k+1})$ denote Clifford multiplication with the volume form of $TX$, and
we let $\KK=c(e_1)\cdots c(e_{4k+1})$ denote Clifford multiplication with the volume form of $E$.
In view of $\JJ^2=-1=\KK^2$ and $\JJ\KK=-\KK\JJ$ these operators generate another quaternionic structure on $S(TX\oplus E)$ which commutes with the one considered previously.
Note that $\II:=\JJ\KK$ is just Clifford multiplication with the volume form of $TX\oplus E$ and satisfies $\II^2=-1$.
Moreover, $\II c(e)=-c(e)\II$, $\JJ c(e)=-c(e)\JJ$, and $\KK c(e)=c(e)\KK$ for all sections $e$ of $E$, while $\II c(Y)=-c(Y)\II$, $\JJ c(Y)=c(Y)\JJ$, and $\KK c(Y)=-c(Y)\KK$, for every vector field $Y$.
Furthermore, $\II^*=-\II$, $\JJ^*=-\JJ$, $\KK^*=-\KK$.
Since the volume elements and Clifford multiplication are parallel, we also have $\nabla\II=\nabla\JJ=\nabla\KK=0$, see \cite[Corollary~II.4.9]{LM89}.
Hence, the Dirac operator $D_E$ commutes with $\JJ$, and anticommutes with $\II$ and $\KK$.
Putting 
$$
S^\pm(TX\oplus E):=\{v\in S(TX\oplus E):\II v=\pm\ii v\},
$$ 
we obtain a decomposition
\begin{equation}\label{E:grading}
S(TX\oplus E)=S^+(TX\oplus E)\oplus S^-(TX\oplus E).
\end{equation}
We consider $S^\pm(TX\oplus E)$ as complex vector bundles where multiplication is given by $\ii$.
Note that $\jj$, $\kk$, $\JJ$, $\KK$, and $c(Y)$ are all of odd degree, with respect to the grading in \eqref{E:grading}, mapping $S^\pm(TX\oplus E)$ into $S^\mp(TX\oplus E)$.
The operator $D_E$ is complex linear, mapping sections of $S^+(TX\oplus E)$ to sections of $S^-(TX\oplus E)$ and vice versa.
Tang and Zhang \cite{TZ02} consider the complex linear skew-adjoint operator $-P_E^*=P_E:=\JJ D_E=D_E\JJ$, and define, cf.\ \cite{AS71},
$$
\ind_2(P_E):=\dim_\C\ker\left(\Gamma(S^+(TX\oplus E))\xrightarrow{P_E}\Gamma(S^+(TX\oplus E))\right)\mod2.
$$

It will be convenient for our purpose to rewrite this index slightly.
To this end, note that $\tau:=\JJ^{-1}\jj$ is a complex antilinear involution which preserves the grading in \eqref{E:grading} and commutes with Clifford multiplication, i.e.\ $\tau^2=1$, $\tau\ii=-\ii\tau$, $\tau\II=-\II\tau$, and $\tau c(Y)=c(Y)\tau$
for every vector field $Y$.
Putting 
$$
S^\pm_\R(TX\oplus E):=\{v\in S^\pm(TX\oplus E):\JJ v=\jj v\},
$$ 
we thus obtain a $\Z_2$-graded real Dirac bundle,
$$
S_\R(TX\oplus E)=S^+_\R(TX\oplus E)\oplus S^-_\R(TX\oplus E),
$$
whose complexification coincides with $S(TX\oplus E)$.
In fact the action of $\jj$ turns $S_\R(TX\oplus E)$ into a $\Z_2$-graded $\Cl_1$-Dirac bundle, see \cite[Definition~II.7.3]{LM89}.
Here we denote the Clifford algebras by $\Cl_q:=\Cl(\R^q)$.
Hence, $\Cl_1\cong\C$.
Since the Dirac operator commutes with $\jj$ and $\JJ$, it restricts to an operator, $D_E\colon\Gamma(S_\R(TX\oplus E))\to\Gamma(S_\R(TX\oplus E))$, which commutes with the $\Cl_1$-action.
Hence, for closed manifolds $X$, its kernel is a finite dimensional $\Z_2$-graded $\Cl_1$-module, giving rise to an analytic index,
$$
\ind(D_E)=[\ker(D_E)]\in\hat{\mathfrak M}_1/i^*\hat{\mathfrak M}_2\cong{\rm KO}^{-1}({\rm pt})\cong\Z_2,
$$
see \cite[Alternative Definition~II.7.4]{LM89}. 
Here $\hat{\mathfrak M}_q$ denotes the Grothendieck group of finite dimensional $\Z_2$-graded $\Cl_q$-modules, and $i^*\colon\hat{\mathfrak M}_{q+1}\to\hat{\mathfrak M}_q$ is the homomorphism
induced by the inclusion, $i\colon\Cl_q\to\Cl_{q+1}$. 
Clearly, this coincides with Tang and Zhang's index, that is,
\begin{equation}\label{E:ind2}
\ind(D_E)=\dim_\R\ker\left(S^+_\R(TX\oplus E)\xrightarrow{D_E^+}S^-_\R(TX\oplus E)\right)=\ind_2(P_E)\in\Z_2.
\end{equation}

Note that in the classical case, $E=TM$, we have $S^\pm_\R(TM\oplus E)=\Lambda_\R^\text{even/odd}T^*M$, $D_E=d+d^*$, and $\ind(D_E)=k(X)$, see \cite{A70} and \cite[Remark~3.6]{TZ02}.

We have the following vanishing result analogous to \cite[Theorem~4.1]{A70}:

\begin{lemma}\label{L:vanish}
Let $E$ be an orientable vector bundle of rank $4k+1$ over an orientable closed manifold $X$ of dimension $4k+1$ such that $w_2(X)=w_2(E)$.
If $E$ admits an orientable subbundle of rank $p\equiv2\mod4$, then $\ind(D_E)=0\in\Z_2$.
\end{lemma}

\begin{proof}
Following Atiyah's argument, we consider Clifford multiplication with the volume form of the $p$-plane bundle, $u:=c(e_1)\cdots c(e_p)$.
Since $p\equiv2\mod4$, we have $u^2=-1$.
Moreover, $u$ commutes with $\ii$, $\jj$, $\kk$, $\II$, $\JJ$, $\KK$, and $c(Y)$, for every vector field $Y$.
In particular, $u$ preserves the grading in \eqref{E:grading} and the real subbundle $S_\R(TX\oplus E)$.
We may assume that the linear connection on $E$ preserves the $p$-plane bundle.
Hence, the volume form of this $p$-plane bundle is parallel and thus $\nabla u=0$.
We conclude that $u$ commutes with $D_E$.
Consequently, $u$ restricts to a complex structure on the kernel of $D_E^+\colon S^+_\R(TX\oplus E)\to S^-_\R(TX\oplus E)$.
Hence, the real dimension of this kernel is even, and $\ind(D_E)=0$, see \eqref{E:ind2}.
\end{proof}

In view of the preceding lemma and \eqref{E:ind2}, the $\Z_2$-index $\ind_2(P_E)$ vanishes for the bundle $E=\xi\oplus\varepsilon^1\oplus\xi$ in the situation of Lemma~\ref{L:spanxixi}.
According to Lemma~\ref{L:tang-zhang} the latter bundle thus has two linearly independent sections.
This proves Lemma~\ref{L:spanxixi} and completes the proof of Theorem~\ref{T:top}.

Let us spell out the following corollary of Tang and Zhang's result, see \cite[Corollary~3.7]{TZ02}, and Lemma~\ref{L:vanish} above:

\begin{corollary}
Let $E$ be an orientable $(4k+1)$-plane bundle over a closed connected orientable $(4k+1)$-manifold $M$ such that $w_2(E)=w_2(M)$.
If $w_{4k}(E)=0$, and $E$ admits an orientable subbundle of rank $p\equiv2\mod4$, then $E$ admits two linearly independent sections.
\end{corollary}

Proceeding as in \cite{AD72} it might be possible to  get rid of the orientability assumption provided $TM\oplus E$ is spinnable, cf.\ \cite[Theorem~2.1]{TZ02}, $w_{4k}(E)=0$, and the complement of the $p$-plane bundle is orientable.

\section{The h-principle}\label{S:h}

In this section we will combine Gromov's h-principle for open manifolds \cite{G86} with Theorem~\ref{T:top}(a) above to establish the following existence result:

\begin{theorem}\label{T:open}
Let $M$ be an open spinnable $5$-manifold and suppose $e\in H^2(M;\Z)$ satisfies $e^2=\frac12p_1(M)$.
Then there exists an oriented rank two distribution of Cartan type $\xi\subseteq TM$ with Euler class $e(\xi)=e$.
\end{theorem}

Suppose $M$ is a smooth $5$-manifold, and let $\mathcal C$ denote the (possibly empty) space of all tangent $2$-plane fields which are of Cartan type.
We consider $\mathcal C$ as a subspace of $\Gamma(\Gr_2(M))$, the space of smooth sections of the Grassmannian bundle of tangent $2$-planes, $\Gr_2(M)\to M$.
To be of Cartan type is a condition on the $2$-jet of a distribution.
Hence, there exists a corresponding (unique) subset $\mathcal R$ in the total space of the $2$-jet bundle, $J^2\Gr_2(M)\to M$, such that
$$
\mathcal C=\bigl\{\xi\in\Gamma(\Gr_2(M)):(j^2\xi)(M)\subseteq\mathcal R\bigr\},
$$
where $j^2\colon\Gamma(\Gr_2(M))\to\Gamma(J^2\Gr_2(M))$ denotes the $2$-jet extension.
We refer to \cite[Chapter~1]{EM02} for basic properties of the jet formalism.
Let $\Gamma(\mathcal R)$ denote the set of all smooth sections of $J^2\Gr_2(M)$ which take values in $\mathcal R$, 
and let $\Gamma_\hol(\mathcal R)$ denote the subset of holonomic solutions, i.e.\ those which are of the form $j^2\xi$ with $\xi\in\Gamma(\Gr_2(M))$.
Consequently, $j^2$ restricts to a homeomorphism
\begin{equation}\label{E:CholR}
\mathcal C\cong\Gamma_\hol(\mathcal R),
\end{equation}
with inverse given by composition with the projection, $q\colon J^2\Gr_2(M)\to\Gr_2(M)$.

Since the differential relation $\mathcal R$ is open and $\Diff(M)$-invariant, the holonomic approximation theorem \cite[Theorem~3.1.1]{EM02} shows 
that $\mathcal R$ satisfies all forms of the local h-principle near any polyhedron of positive codimension, see \cite[Section~7.2.1]{EM02}.
For open manifolds this implies the global h-principle, see \cite[Section~7.2.2]{EM02} or \cite[Section~2.2.2]{G86}.
More precisely, we have the following special case of a result of Gromov's:

\begin{lemma}[Gromov]\label{L:gromov}
If $M$ is an open $5$-manifold, then the differential relation $\mathcal R$ satisfies the parametric h-principle.
In particular, the inclusion $\Gamma_\hol(\mathcal R)\to\Gamma(\mathcal R)$ is a weak homotopy equivalence.
\end{lemma}

We do not know if $\mathcal R$ satisfies the (parametric) h-principle on closed $5$-manifolds.
In particular, it is unclear to what extent the relative h-principle holds for the pair $(D^5,S^4)$.

Recall that a rank two distribution $\xi\subseteq TM$ of Cartan type gives rise to a filtration of $TM$ by subbundles, and
the Lie bracket induces the structure of a graded nilpotent Lie algebra on every fiber of the associated graded, $\gr(TM)$.
The Lie algebras $\gr_x(TM)$ are all isomorphic to the $5$-dimensional graded nilpotent Lie algebra with basis: $X$, $Y$, $[X,Y]$, $[X,[X,Y]]$, $[Y,[X,Y]]$.
Thus, a rank two distribution of Cartan type provides a reduction of the structure group of $TM$ to
\begin{equation}\label{E:H}
H:=\left\{\left(\begin{smallmatrix}A&*&*\\0&\det(A)&*\\0&0&\det(A)A\end{smallmatrix}\right):A\in\GL(2)\right\}\subseteq\GL(5).
\end{equation}
Equivalently, $\xi$ provides a section of the bundle $P/H=P\times_{\GL(5)}(\GL(5)/H)$ associated to the frame bundle $P$ of $M$. 
Since the Lie algebra $\gr_x(TM)$ only depends on the $2$-jet of $\xi$ at $x\in M$, this is in fact induced from a bundle map,
\begin{equation}\label{E:levi}
\mathcal L\colon\mathcal R\to P/H.
\end{equation}

\begin{lemma}\label{L:jets}
The map $\mathcal L$ is a smooth (locally trivial) fiber bundle with typical fiber diffeomorphic to $\R^{113}$. 
In particular, $\mathcal L$ admits a global smooth cross section, and induces a homotopy equivalence $\Gamma(\mathcal R)\to\Gamma(P/H)$.
\end{lemma}

\begin{proof}
Since the statement is local in $M$, we may restrict our attention to open subsets $M\subseteq\R^5$.
Using the coordinate frame we identify
$$
\Gr_2(M)=M\times\Gr_{2,5},\quad J^2\Gr_2(M)=J^2(M,\Gr_{2,5}),\quad P/H=M\times\GL(5)/H.
$$
Here $\Gr_{2,5}$ denotes the Grassmannian of $2$-planes in $\R^5$.
Moreover, $J^2\Gr_2(M)$ denotes the space of $2$-jets of sections of the (trivial) bundle $\Gr_2(M)\to M$ which is being identified with $J^2(M,\Gr_{2,5})$, the space of $2$-jets of maps from $M$ to $\Gr_{2,5}$.

Let $U_0\subseteq\Gr_{2,5}$ denote the open subset of $2$-planes in $\R^5$ which are transverse to the subspace spanned by the unit vectors $e_3,e_4,e_5$,
and let $U\subseteq\Gr_2(M)$ denote the open subset corresponding to $M\times U_0$.
Note that the subgroup
$$
G_0:=\left\{\begin{pmatrix}I_2&0\\B&A\end{pmatrix}:B\in\R^{3\times2},A\in\GL(3)\right\}\subseteq\GL(5)
$$
acts transitively on $p_0^{-1}(U_0)\subseteq\GL(5)/H$, where $p_0\colon\GL(5)/H\to\Gr_{2,5}$ denotes the canonical projection.
Denoting the stabilizer by $N:=G_0\cap H$, we obtain $p_0^{-1}(U_0)=G_0/N$.
Since $N$ is diffeomorphic to $\R^2$, the (principle) bundle $G_0\to G_0/N$ is trivializable.
Taking the product with $M$, we obtain a trivializable smooth fiber bundle with typical fiber $N\cong\R^2$,
$$
\pi\colon M\times G_0\to M\times G_0/N=M\times p_0^{-1}(U_0)=p^{-1}(U).
$$
Here $p\colon P/H\to\Gr_2(M)$ denotes the canonical projection.

We identify $U_0=\R^{3\times2}$ by assigning to a matrix $(a,b)\in\R^{3\times2}$ with columns in $a,b\in\R^3$ the $2$-plane spanned by $(1,0,a^t)^t$ and $(0,1,b^t)^t$. 
Correspondingly, we obtain identifications
$$
U=M\times\R^{3\times2}\qquad\text{and}\qquad q^{-1}(U)=J^2(M,\R^{3\times2}),
$$
where $q\colon J^2\Gr_2(M)\to\Gr_2(M)$ denotes the canonical projection.

Note that $q^{-1}(U)\cap\mathcal R=\mathcal L^{-1}(p^{-1}(U))$, and let
$$
\tilde{\mathcal L}\colon q^{-1}(U)\cap\mathcal R\to p^{-1}(U)
$$
denote the restriction of $\mathcal L$.

Every distribution in $\Gamma(U)\subseteq\Gamma(\Gr_2(M))$ is spanned by two vector fields
\begin{equation}\label{E:ab}
X=\begin{pmatrix}1\\0\\a\end{pmatrix}\qquad\text{and}\qquad Y=\begin{pmatrix}0\\1\\b\end{pmatrix},
\end{equation}
where $a,b\in C^\infty(M,\R^3)$. A straight forward calculation shows that 
\begin{equation}\label{E:cde}
[X,Y]=\begin{pmatrix}0\\0\\c\end{pmatrix},\qquad [X,[X,Y]]=\begin{pmatrix}0\\0\\d\end{pmatrix},\qquad[Y,[X,Y]]=\begin{pmatrix}0\\0\\e\end{pmatrix},
\end{equation}
where $c,d,e\in C^\infty(M,\R^3)$ are of the form:
\begin{align}
\notag
c(x)&=b_1(x)-a_2(x)+\bar\phi_3\bigl(a(x),b(x),\bar D_xa,\bar D_xb\bigr)=:\phi_3(j_x^2(a,b)),
\label{E:phi345}\\
d(x)&=b_{11}(x)-a_{12}(x)+\bar\phi_4\bigl(j_x^1a,j_x^1b,\bar D^2_xa,\bar D_x^2b\bigr)=:\phi_4(j_x^2(a,b)),
\\\notag
e(x)&=b_{12}(x)-a_{22}(x)+\bar\phi_5\bigl(j_x^1a,j^1_xb,\bar D_x^2a,\bar D_x^2b\bigr)=:\phi_5(j_x^2(a,b)).
\end{align}
Here $\bar\phi_i$ depend on the following parts of the $2$-jet as indicated: $a_i:=\frac{\partial a}{\partial x^i}$, $a_{ij}:=\frac{\partial^2a}{\partial x^i\partial x^j}$, $\bar D_xa=(a_3(x),a_4(x),a_5(x))$, $\bar D_x^2a:=(a_{ij}(x))_{1\leq i\leq j\leq 5,3\leq j}$, and similarly for $b$.
Consider the smooth map
\begin{equation}\label{E:phi}
\phi\colon J^2(M,\R^{3\times2})\to M\times\R^{3\times2}\times\R^{3\times3},\quad \phi:=(\phi_0,\phi_1,\phi_2,\phi_3,\phi_4,\phi_5),
\end{equation}
where $\phi_3,\phi_4,\phi_5$ are defined in \eqref{E:phi345}, and $\phi_0,\phi_1,\phi_2$ denote the components of the standard projection, $J^2(M,\R^{3\times2})\to M\times\R^{3\times2}$, that is:
\begin{align*}
x&=:\phi_0(j_x^2(a,b)),\\
a(x)&=:\phi_1(j_x^2(a,b)),\\
b(x)&=:\phi_2(j_x^2(a,b)).
\end{align*} 
Thus, $\phi$ is the unique map such that for all $a,b\in C^\infty(M,\R^3)$ and all $x\in M$,
\begin{equation}\label{E:abcde}
\phi(j^2_x(a,b))=\bigl(x,a(x),b(x),c(x),d(x),e(x)\bigr).
\end{equation}
Clearly, $X$, $Y$, $[X,Y]$, $[X,[X,Y]]$, $[Y,[X,Y]]$ span the tangent space at $x\in M$ iff $c(x),d(x),e(x)$ span $\R^3$.
Hence, we obtain the following local description of $\mathcal R$,
$$
q^{-1}(U)\cap\mathcal R=\phi^{-1}\bigl(M\times\R^{3\times2}\times\GL(3)\bigr).
$$
We let 
$$
\tilde\phi\colon q^{-1}(U)\cap\mathcal R\to M\times\R^{3\times2}\times\GL(3)
$$
denote the restriction of $\phi$.

These maps and identifications are summarized in the following commutative diagram.
The lower central rectangle commutes in view of \eqref{E:ab}, \eqref{E:cde}, and \eqref{E:abcde}, using the obvious identification $\R^{3\times2}\times\GL(3)=G_0$.
$$
\xymatrix@C=3ex{
J^2\Gr_2(M)\ar@/^5ex/[rrr]^-q&\mathcal R\ar@{_{(}->}[l]\ar[r]^-{\mathcal L}&P/H\ar[r]^-p&\Gr_2(M)
\\
q^{-1}(U)\ar@{^{(}->}[u]\ar@{=}[d]&q^{-1}(U)\cap\mathcal R\ar@{_{(}->}[l]\ar[r]^-{\tilde{\mathcal L}}\ar@{=}[d]\ar@{^{(}->}[u]&p^{-1}(U)\ar@{^{(}->}[u]\ar[r]\ar@{=}[d]&U\ar@{^{(}->}[u]\ar@{=}[d]
\\
J^2(M,\R^{3\times2})\ar[d]^-\phi&\phi^{-1}\bigl(M\times\R^{3\times2}\times\GL(3)\bigr)\ar@{_{(}->}[l]\ar[d]^-{\tilde\phi}&M\times G_0/N\ar[r]&M\times\R^{3\times2}
\\
M\times\R^{3\times2}\times\R^{3\times3}&M\times\R^{3\times2}\times\GL(3)\ar@{=}[r]\ar@{_{(}->}[l]&M\times G_0\ar[u]_-\pi
}
$$

We will next show that the map $\phi$ in \eqref{E:phi} is a trivializable smooth fiber bundle with typical fiber diffeomorphic to $\R^{111}$.
To see this, consider the map
$$
\psi\colon J^2(M,\R^{3\times2})\to F:=\R^{3\times5}\times\R^{3\times15}\times\R^{3\times4}\times\R^{3\times13}\cong\R^{111}
$$
defined by
$$
\psi(j^2_x(a,b)):=\Bigl((a_i(x))_{1\leq i\leq5},(a_{ij}(x))_{1\leq i\leq j\leq5},(b_i(x))_{2\leq i\leq5},(b_{ij}(x))_{1\leq i\leq j\leq5,4\leq i+j}\Bigr)
$$
where we are denoting the partial derivatives as above.
Then the map
$$
(\phi,\psi)\colon J^2(M,\R^{3\times2})\to M\times\R^{3\times2}\times\R^{3\times3}\times F
$$
is a diffeomorphism.
Indeed, the formulas in \eqref{E:phi345} permit to write down the inverse mapping explicitly.
Given $\phi(j^2_x(a,b))$ and $\psi(j^2_x(a,b))$, the first equation in \eqref{E:phi345} permits to compute $b_1(x)$ and then, using the other two equations, one obtains $b_{11}(x)$ and $b_{12}(x)$.
Hence, the full $2$-jet $j^2_x(a,b)$ can be recovered from $\phi(j^2_x(a,b))$ and $\psi(j^2_x(a,b))$.
One readily checks that this provides a smooth inverse to $(\phi,\psi)$.
This shows that the map $\phi$ is indeed a trivializable fiber bundle with typical fiber $F\cong\R^{111}$.
We conclude that the restriction $\tilde\phi$ of $\phi$ is a trivializable fiber bundle with typical fiber $\R^{111}$ too.

As composition of the trivializable bundle $\tilde\phi$ with typical fiber $F\cong\R^{111}$ and the trivializable bundle $\pi$ with typical fiber $N\cong\R^2$, the map $\tilde{\mathcal L}$ is a trivializable bundle with typical fiber $N\times F\cong\R^{113}$.
Hence, $\mathcal L$ is locally trivial with typical fiber $\R^{113}$.
In particular, $\mathcal L$ admits a global smooth section $\sigma\colon P/H\to\mathcal R$, see for instance \cite[Theorem~I.5.7]{KN63}.
Since $\mathcal L$ has contractible fibers, $\sigma\circ\mathcal L$ is homotopic to the identity through a smooth homotopy preserving the fibers of $\mathcal L$.
Hence, the map $\Gamma(P/H)\to\Gamma(\mathcal R)$ induced by $\sigma$ is homotopy inverse to the map $\Gamma(\mathcal R)\to\Gamma(P/H)$ induced by $\mathcal L$.
\end{proof}

We conclude this section with a proof of Theorem~\ref{T:open}.
To this end, let $M$ be an open spinnable $5$-manifold and suppose $e\in H^2(M;\Z)$ satisfies $e^2=\frac12p_1(M)$.
Let $\xi$ denote the oriented $2$-plane bundle over $M$ with Euler class $e(\xi)=e$.
In view of Theorem~\ref{T:top}(a), and since $\Lambda^2\xi\cong\varepsilon^1$, we have $TM\cong\xi\oplus\Lambda^2\xi\oplus(\xi\otimes\Lambda^2\xi)$.
Hence, there exists a section of $P/H\to M$ with underlying $2$-plane bundle isomorphic to $\xi$.
This can be lifted over $\mathcal L$ to a section of $\mathcal R\to M$, for the bundle $\mathcal L\colon\mathcal R\to P/H$ admits a global section in view of Lemma~\ref{L:jets}.
According to the h-principle, see Lemma~\ref{L:gromov}, the latter is homotopic to a holonomic section in $\Gamma_\hol(\mathcal R)$.
Consequently, see \eqref{E:CholR}, there exists a rank two distribution of Cartan type on $M$ with underlying $2$-plane bundle isomorphic to $\xi$.
By construction, this tangent $2$-plane bundle has Euler class $e$.

\section{Principal connections of Cartan type}\label{S:princ}

Let $\goe$ be a real three dimensional Lie algebra.
A connection $1$-form $\omega\in\Omega^1(\Sigma;\goe)$ on a surface $\Sigma$ is called of Cartan type if for one (and then every) Lie group $G$ with Lie algebra $\goe$ the corresponding principal connection, regarded as a right $G$-invariant tangent $2$-plane field on $\Sigma\times G$, is of Cartan type.
More explicitly, the $2$-plane field on $\Sigma\times G$ we associate with the $1$-form $\omega$ is spanned by the vector fields of the form $X+\omega(X)$, where $X\in\mathfrak X(\Sigma)$ is a vector field on the surface, and $\omega(X)\in C^\infty(\Sigma;\goe)$ is regarded as a right $G$-invariant vertical vector field.

In this section we will consider the $3$-dimensional abelian Lie algebra $\aoe$ and the $3$-dimensional Heisenberg algebra $\hoe$. 
We will use connection $1$-forms with values in these Lie algebras to equip some closed $5$-manifolds with tangent $2$-plane fields of Cartan type.
More precisely, if $N$ is a closed orientable $3$-manifold and if $\Sigma$ is a closed connected (possibly non-orientable) surface with Euler characteristic $\chi(\Sigma)\geq-1$, then $\Sigma\times N$ admits a tangent $2$-plane field of Cartan type which is transverse to the fibers $*\times N$, see Examples~\ref{Ex:1} and \ref{Ex:2} below.
Recall that this class of surfaces is comprised of: the sphere $S^2$; the projective space $\RP^2$; the torus $T^2$; and the Klein bottle $K$.

We are not able to construct Cartan type distributions on $\Sigma\times N$ if $\chi(\Sigma)<-1$, despite the fact that there are no topological obstructions.
More explicitly, for every closed (possibly non-orientable) surface $\Sigma$ and every closed orientable $3$-manifold $N$ there exists an isomorphism of vector bundles as in \eqref{E:TMdecom} where $M=\Sigma\times N$, $\xi=\pr_1^*T\Sigma$ and $\pr_1\colon M\to\Sigma$ denotes the canonical projection.
Indeed, $TM\cong\xi\oplus\varepsilon^3$ since closed orientable $3$-manifolds are parallelizable according to a theorem of Stiefel.
Moreover, combining Wu's formula $w_2(\Sigma)=w_1(\Sigma)^2$ with well known formulas for Stiefel--Whitney classes\footnote{For every $2$-plane bundle $\xi$ we have $w(\Lambda^2\xi)=1+w_1(\xi)$, $w(\xi\otimes\Lambda^2\xi)=1+w_1(\xi)+w_2(\xi)$, and thus $w(\Lambda^2\xi\oplus(\xi\otimes\Lambda^2\xi))=1+w_1(\xi)^2+w_2(\xi)+w_1(\xi)w_2(\xi)$.}, we see that  all Stiefel--Whitney classes of $\Lambda^2T\Sigma\oplus(T\Sigma\otimes\Lambda^2T\Sigma)$ are trivial.
Hence the structure group of the latter bundle can be reduced to $\Spin(3)$. Consequently, this bundle is trivializable and we obtain $\Lambda^2\xi\oplus(\xi\otimes\Lambda^2\xi)\cong\varepsilon^3$, whence \eqref{E:TMdecom}.

On the other hand, we will show that $\aoe$-valued connection $1$-forms of Cartan type are subject to strong geometric restrictions. 
The only orientable closed connected surface admitting an $\aoe$-valued connection $1$-form of Cartan type is the sphere, see Lemma~\ref{L:aoe}(a) below.
Moreover, it is not always possible to extend an $\aoe$-valued connection $1$-form of Cartan type defined in a neighborhood of the unit circle across the disk, see Theorem~\ref{T:extension} and Corollary~\ref{C:ext} below.

\begin{lemma}\label{L:aoe}
Let $\aoe$ denote the $3$-dimensional abelian Lie algebra.

a) A connection $1$-form $\omega\in\Omega^1(\Sigma;\aoe)$ on a surface $\Sigma$ is of Cartan type iff the curvature $d\omega\in\Omega^2(\Sigma;\aoe)$ is nowhere vanishing and the corresponding map into projective space, $\Sigma\to P\aoe:=(\aoe\setminus0)/\,\R^\times\cong\RP^2$, is an immersion. 
In particular, a closed connected surfaces admitting such a connection $1$-form of Cartan type is diffeomorphic to the $2$-dimensional sphere or projective space.

b) Let $N$ be an orientable closed $3$-manifold and let $\Sigma$ be a compact surface which admits an $\aoe$-valued connection $1$-form of Cartan type.
Then $\Sigma\times N$ admits a tangent $2$-plane field of Cartan type which is transverse to the fibers $*\times N$ and may be chosen arbitrarily close to the flat connection with leaves $\Sigma\times*$ with respect to the $C^\infty$-topology.
\end{lemma}

\begin{proof}
Let $A$ be a Lie group with Lie algebra $\aoe$.
Choose local coordinates on $\Sigma$ and consider the corresponding coordinate vector fields, $\partial_x:=\frac\partial{\partial x}$ and $\partial_y:=\frac\partial{\partial y}$.
Interpreting $\omega(\partial_x)$ and $\omega(\partial_y)$ as right invariant vertical vector fields on $\Sigma\times A$, the corresponding principal connection is locally spanned by the two vector fields
$$
X=\partial_x+\omega(\partial_x)
\quad\text{and}\quad
Y=\partial_y+\omega(\partial_y).
$$
Writing the curvature as $d\omega=F\,dx\,dy$ and considering $F\colon\Sigma\to\aoe$ as a right invariant vertical vector field on $\Sigma\times A$, we have:
$$
[X,Y]=F,\qquad
[X,[X,Y]]=\frac{\partial F}{\partial x},
\quad\text{and}\quad
[Y,[X,Y]]=\frac{\partial F}{\partial y}.
$$
Hence, $\omega$ is of Cartan type iff $F$, $\frac{\partial F}{\partial x}$, and $\frac{\partial F}{\partial y}$ are linearly independent in $\aoe$, at each point in $\Sigma$.
Equivalently, this is the case iff the (locally defined) $F\colon\Sigma\to\aoe$ is nowhere vanishing, and defines an immersion into the projective space $P\aoe$. 
To conclude the proof of (a), recall that every immersion from a closed surface into $\RP^2$ is a covering by Ehresmann's fibration theorem.

To see (b), recall Stiefel's theorem asserting that every orientable closed $3$-manifold is parallelizable, see \cite[Problem~12-B]{MS}.
Hence, there exists a global frame of vector fields $Z_1$, $Z_2$, $Z_3$ on $N$.
Let $\omega\in\Omega^1(\Sigma;\aoe)$ be a connection $1$-form of Cartan type,
and let $\omega^i\in\Omega^1(\Sigma;\R)$ denote its components with respect to some basis of $\aoe$.
For $\varepsilon>0$ consider the $1$-form $\varepsilon\sum_{i=1}^3Z_i\omega^i$ taking values in the Lie algebra of vector fields $\mathfrak X(N)$.
The corresponding connection on the bundle $\Sigma\times N\to\Sigma$ is a $2$-plane bundle $\xi_\varepsilon$ transverse to the fibers $*\times N$.
We will show that this $2$-plane bundle is of Cartan type, provided $\varepsilon$ is sufficiently small.
If $\partial_x=\frac\partial{\partial x}$ and $\partial_y=\frac\partial{\partial y}$ are local coordinate vector fields on $\Sigma$, then the $2$-plane field $\xi_\varepsilon$ is locally spanned by the two vector fields
$$
X=\partial_x+\varepsilon A\qquad\text{and}\qquad
Y=\partial_y+\varepsilon B
$$
where $A:=\sum_{i=1}^3Z_i\omega^i(\partial_x)$ and $B:=\sum_{i=1}^3Z_i\omega^i(\partial_y)$ are considered as vertical vector fields on $\Sigma\times N$.
Writing $F:=\frac{\partial B}{\partial x}-\frac{\partial A}{\partial y}$ we find, as $\varepsilon\to0$:
\begin{align}
\notag
[X,Y]&=\varepsilon F+O(\varepsilon^2),
\\\label{E:tyu}
[X,[X,Y]]&=\varepsilon\tfrac{\partial F}{\partial x}+O(\varepsilon^2)
\\\notag
[Y,[X,Y]]&=\varepsilon\tfrac{\partial F}{\partial y}+O(\varepsilon^2)
\end{align}
Since $\omega$ is of Cartan type, and since $Z_1,Z_2,Z_3$ is a global frame of $N$, the vector fields $F$, $\frac{\partial F}{\partial x}$, $\frac{\partial F}{\partial y}$ constitute a frame of the vertical bundle of $\Sigma\times N\to\Sigma$. 
Using \eqref{E:tyu}, we conclude that $[X,Y]$, $[X,[X,Y]]$, and $[Y,[X,Y]]$ form a frame of the vertical bundle, and thus $\xi_\varepsilon$ is of Cartan type, provided $\varepsilon>0$ is sufficiently small.
These considerations hold true over every compact subset in the domain of the local coordinates on $\Sigma$.
Using the compactness of $\Sigma$, we immediately obtain the global statement.
\end{proof}

\begin{example}[$S^2\times N$]\label{Ex:1}
With respect to a basis of $\aoe$, the expression
\begin{equation}\label{E:s2aoe}
\omega:=\begin{pmatrix}x\\y\\z\end{pmatrix}\times\begin{pmatrix}dx\\dy\\dz\end{pmatrix}
=\begin{pmatrix}ydz-zdy\\zdx-xdz\\xdy-ydx\end{pmatrix}
\end{equation}
defines an $\aoe$-valued connection $1$-form of Cartan type on the unit sphere $S^2\subseteq\R^3$. 
Indeed, this follows from Lemma~\ref{L:aoe}(a), the curvature is of the form $d\omega=F\Omega$ where 
$$
F=2\begin{pmatrix}x\\y\\z
\end{pmatrix},
\quad\text{and}\quad
\Omega=xdydz+ydzdx+zdxdy=\frac{dydz}x=\frac{dzdx}y=\frac{dxdy}z
$$ 
denotes the standard volume form on $S^2$.
Using Lemma~\ref{L:aoe}(b), we conclude that $S^2\times N$ admits a $2$-plane  field of Cartan type which is transverse to the fibers $*\times N$, for every closed orientable $3$-manifold $N$.
Clearly, the Euler class of such a $2$-plane field is Poincar\'e dual to the fibers $*\times N$.
Note that the $1$-form \eqref{E:s2aoe} is invariant under the antipodal map and thus descends to a connection $1$-form on the projective space.
Consequently, $\RP^2\times N$ also admits a $2$-plane field of Cartan type which is transverse to the fibers $*\times N$ for every closed orientable $3$-manifold $N$.
\end{example}

Let $D$ denote the closed unit disk in $\R^2$.

\begin{theorem}\label{T:extension}
Let $\aoe$ denote the $3$-dimensional abelian Lie algebra, and consider the canonical projection $\pi\colon\aoe\setminus0\to\tilde P\aoe:=(\aoe\setminus0)/\R^+\cong S^2$.
Suppose $\omega$ is an $\aoe$-valued connection $1$-form of Cartan type defined in a neighborhood of the unit circle $\partial D\subseteq\R^2$ and let $F$ denote the (nowhere vanishing) $\aoe$-valued smooth function such that $d\omega=Fdxdy$.
Suppose the immersion $\pi F\colon\partial D\to\tilde P\aoe$ is injective, and let $U$ denote the connected component of $\tilde P\aoe\setminus\pi F(\partial D)$ corresponding to the inward pointing radial vector field along $\partial D$.
Then the germ of $\omega$ along $\partial D$ can be extended to a connection $1$-form of Cartan type across $D$ if and only if $\int_{\partial D}\omega$ is contained in the convex hull of $\pi^{-1}(U)$.
\end{theorem}

\begin{proof}
Suppose $\int_{\partial D}\omega$ is contained in the convex hull of $\pi^{-1}(U)$.
By the differentiable Sch\"onflies theorem \cite{BK87} the embedding $\pi F\colon\partial D\to\tilde P\aoe$ extends to an embedding of the disk mapping the open unit ball $B$ onto $U$.
Hence, the germ of $F$ along $\partial D$ can be extended to a smooth map $\tilde F\colon D\to\aoe\setminus0$ such that $\pi\tilde F\colon D\to\tilde P\aoe$ is an immersion
and such that the convex cone spanned by $\tilde F(B)$ coincides with the convex hull of $\pi^{-1}(U)$.
According to Lemma~\ref{L:4} below, there exists a positive smooth function $\lambda$ on $D$ such that $\lambda=1$ near $\partial D$, and
$\int_D\lambda\tilde F dxdy=\int_{\partial D}\omega$.
Let $\tilde\omega_0$ be any $1$-form on $D$ which coincides with $\omega$ near $\partial D$.
Then $\lambda\tilde Fdxdy-d\tilde\omega_0$ vanishes in a neighborhood of $\partial D$ and $\int_D\lambda\tilde Fdxdy-d\tilde\omega_0=0$ by Stokes' theorem.
Since the integral induces an isomorphism $H^2(D,\partial D;\R)\cong\R$, we conclude that there exists a $1$-form $\tilde\omega_1$ on $D$ such that $d\tilde\omega_1=\lambda\tilde Fdxdy-d\tilde\omega_0$ and $\tilde\omega_1=0$ near $\partial D$.
Hence, $\tilde\omega:=\tilde\omega_0+\tilde\omega_1$ extends the germ of $\omega$ across $D$ and $d\tilde\omega=\lambda\tilde Fdxdy$.
Clearly, $\pi\lambda\tilde F=\pi\tilde F$, and thus $\tilde\omega$ is of Cartan type, see Lemma~\ref{L:aoe}(a).

To see the other implication, suppose $\tilde\omega$ is a connection $1$-form of Cartan type on $D$ which coincides with $\omega$ in a neighborhood of $\partial D$.
Writing $d\tilde\omega=\tilde Fdxdy$, the map $\pi\tilde F\colon D\to\tilde P\aoe$ is an immersion, see Lemma~\ref{L:aoe}(a).
Actually, this map is an embedding since the restriction to $\partial D$ was assumed to be injective.
Indeed, considering $D$ as the upper hemisphere in the sphere $S^2$, we may use the differentiable Sch\"onflies theorem to extend the map $\pi\tilde F$ across the lower hemisphere to an immersion $S^2\to\tilde P\aoe$, and conclude that this extension has to be an embedding according to Ehresmann's fibration theorem.
We conclude $\pi\tilde F(B)\subseteq U$, and thus the convex cone spanned by $\tilde F(B)$ is contained in the convex hull of $\pi^{-1}(U)$.
Using Stokes' theorem and Lemma~\ref{L:4} below, we see that $\int_{\partial D}\omega=\int_D\tilde Fdxdy$ is contained in the convex hull of $\pi^{-1}(U)$.
\end{proof}

\begin{lemma}\label{L:4}
Consider a smooth map $f\colon D^n\to\R^{n+1}\setminus0$ such that $f/|f|\colon D^n\to S^n$ is an immersion. 
Then the set
\begin{equation}\label{E:44}
\left.\left\{\int_{D^n}\lambda fdx^1\cdots dx^n\right|\text{$\lambda\in C^\infty(D^n)$, $\lambda>0$, and $\lambda=1$ near $\partial D^n$}\right\}
\end{equation}
coincides with the convex cone spanned by $f(B^n)$ in $\R^{n+1}$.
Here $B^n$ denotes the open unit ball in $\R^n$, and $D^n$ denotes the closed unit disk in $\R^n$.
\end{lemma}

\begin{proof}
Let $\Lambda$ denote the set in \eqref{E:44} and write $C$ for the convex cone spanned by $f(B^n)$.
Clearly, both sets are convex.
Moreover, $\Lambda$ and $C$ are both open in $\R^{n+1}$ since $f/|f|$ is assumed to be an immersion.
Furthermore, $\Lambda$ and $C$ are dense subsets of the convex cone spanned by $f(D^n)$.
Since any two open convex sets with the same closure coincide, see \cite[Theorem~6.3]{R70}, we conclude $\Lambda=C$.
\end{proof}

\begin{corollary}\label{C:ext}
There exists an $\aoe$-valued connection $1$-form $\omega$ of Cartan type on the punctured plane $\R^2\setminus0$ which has the following properties:
\begin{enumerate}[(a)]
\item
There does not exist an $\aoe$-valued connection $1$-form of Cartan type on $\R^2$ that coincides with $\omega$ on $\R^2\setminus B$.
\item
For every $\varepsilon>0$ there exists an $\aoe$-valued connection $1$-form of Cartan type on $\R^2$ which coincides with $\omega$ on $\R^2\setminus B_{1+\varepsilon}$. 
\end{enumerate}
Here $B_r\subseteq\R^2$ denotes the open ball of radius $r$ centered at the origin, and $B:=B_1$ denotes the open unit ball.
\end{corollary}

\begin{proof}
We identify the punctured plane with the twice punctured unit sphere, $\R^2\setminus0=S^2\cap\{-1<z<1\}$, and $\R^2=S^2\cap\{-1<z\}$.
Fix a constant $\alpha<-1$.
With respect to some fixed basis of $\aoe$, the following expression defines an $\aoe$-valued $1$-form on the twice punctured sphere $S^2\cap\{-1<z<1\}$:
$$
\omega
:=\begin{pmatrix}ydz-zdy\\zdx-xdz\\xdy-ydx\end{pmatrix}
+\frac\alpha{x^2+y^2}\begin{pmatrix}
0\\
0\\
xdy-ydx
\end{pmatrix}
$$
Note the second term is closed and thus $d\omega$ extends smoothly across the punctures, inducing the identical diffeomorphism $F\colon S^2\to\tilde P\aoe=S^2$, cf.\ Example~\ref{Ex:1}.
For $-1<h<1$ we let $D_h$ denote the hemisphere $S^2\cap\{h\leq z\}$.
Note that 
$$
\int_{\partial D_h}\omega=2\pi\begin{pmatrix}0\\0\\1-h^2+\alpha\end{pmatrix}
$$
and thus the third component of this integral is strictly negative, for all $h$.
For $0\leq h<1$ the convex cone spanned by $F(D_h)$ is contained in the half space of vectors with non-negative third component, hence, according to Theorem~\ref{T:extension}, there does not exist a connection $1$-form of Cartan type on $S^2\cap\{-1<z\}$ which coincides with $\omega$ on $S^2\cap\{-1<z\leq h\}$.
This shows (a).
For $h<0$, however, the convex cone spanned by $F(D_h)$ coincides with $\aoe$ and thus there exists a connection $1$-form of Cartan type on $S^2\cap\{-1<z\}$ which coincides with $\omega$ on $S^2\cap\{-1<z\leq h\}$ in view of Theorem~\ref{T:extension}.
This shows (b).
\end{proof}

\begin{lemma}\label{L:hoe}
Let $\hoe$ denote the $3$-dimensional Heisenberg algebra, suppose $\Sigma$ is a compact surface which admits an $\hoe$-valued connection $1$-form of Cartan type, and consider a closed orientable $3$-manifold $N$.
Then $\Sigma\times N$ admits a $2$-plane field of Cartan type which is transverse to the fibers $*\times N$ and may be chosen arbitrarily close to the flat connection with leaves $\Sigma\times*$ with respect to the $C^\infty$-topology.
\end{lemma}

\begin{proof}
Recall that every orientable closed $3$-manifold is parallelizable according to Stiefel's theorem.
Using the Lutz--Martinet theorem, see \cite[Theorem~4.3.1]{G08}, we see that there exists a contact structure on $N$ whose contact plane field, $E$, is trivializable.
We fix a splitting $TN\cong\gr(TN):=E\oplus TN/E$, and let $\rho_\varepsilon$ denote the vector bundle automorphism which scales by $\varepsilon$ on $E$ and $\varepsilon^2$ on $TN/E$.
Denoting the Levi bracket on $\gr(TN)$ by $\{-,-\}$, we have
\begin{equation}\label{E:fghj}
\rho_\varepsilon^{-1}[\rho_\varepsilon V,\rho_\varepsilon W]=\{V,W\}+O(\varepsilon),\qquad\text{as $\varepsilon\to0$,}
\end{equation}
for any two vector fields $V$ and $W$ on $N$.

Recall that every connection on the bundle $\Sigma\times N\to\Sigma$ can be described by a $1$-form $\omega\in\Omega^1(\Sigma;\mathfrak X(N))$ taking values in the Lie algebra of vector fields.
For $\varepsilon>0$, consider the scaled connection $1$-form $\omega_\varepsilon:=\rho_\varepsilon\omega$, and let $\xi_\varepsilon$ denote the corresponding $2$-plane field on $\Sigma\times N$.
If $\partial_x$ and $\partial_y$ are local coordinate vector fields on $\Sigma$, then the $2$-plane field $\xi_\varepsilon$ is locally spanned by the two vector fields $X_\varepsilon=\partial_x+\rho_\varepsilon A$ and $Y_\varepsilon=\partial_y+\rho_\varepsilon B$, where $A:=\omega(\partial_x)$ and $B:=\omega(\partial_y)$.
Putting $F:=\frac{\partial B}{\partial x}-\frac{\partial A}{\partial y}+\{A,B\}$, and using \eqref{E:fghj}, one readily verifies:
\begin{align}
\notag
\rho_\varepsilon^{-1}[X_\varepsilon,Y_\varepsilon]&=F+O(\varepsilon)
\\\label{E:aassdd}
\rho_\varepsilon^{-1}[X_\varepsilon,[X_\varepsilon,Y_\varepsilon]]
&=\tfrac{\partial F}{\partial x}+\{A,F\}+O(\varepsilon)
\\\notag
\rho_\varepsilon^{-1}[Y_\varepsilon,[X_\varepsilon,Y_\varepsilon]]
&=\tfrac{\partial F}{\partial y}+\{B,F\}+O(\varepsilon)
\end{align}

Suppose $\omega$ is obtained from an $\hoe$-valued connection $1$-form of Cartan type via 
$$
\hoe\subseteq C^\infty(N,\hoe)\cong\Gamma(\gr(TN))\cong\mathfrak X(N)
$$
where the first inclusion is by constant functions, and the middle isomorphism of Lie algebras is induced by some trivialization of bundles of Lie algebras $N\times\hoe\cong\gr(TN)$.
Then $F$, $\tfrac{\partial F}{\partial x}+\{A,F\}$, and $\tfrac{\partial F}{\partial y}+\{B,F\}$ are pointwise linearly independent in $\hoe$.
Using \eqref{E:aassdd}, we conclude that the $2$-plane field $\xi_\varepsilon$ is of Cartan type, provided $\varepsilon$ is sufficiently small.
These considerations hold true over every compact subset in the domain of the local coordinates on $\Sigma$.
Using the compactness of $\Sigma$, we immediately obtain the global statement.
\end{proof}

\begin{example}[$T^2\times N$]\label{Ex:2}
Let $E,F,[E,F]$ be a basis of the Heisenberg algebra $\hoe$.
The expression
$$
\omega=\bigl(E\cos(x)+F\sin(x)\bigr)dy
$$
defines an $\hoe$-valued connection $1$-form of Cartan type on $\R^2$ which is periodic and descends to a connection $1$-form of Cartan type on the torus $T^2=\R^2/(2\pi\Z)^2$.
Using Lemma~\ref{L:hoe}, we conclude that $T^2\times N$ admits a tangent $2$-plane field of Cartan type which is transverse to the fibers $*\times N$,
for every closed orientable $3$-manifold $N$.
The Euler class of this $2$-plane bundle is trivial.
Note that $\omega$ is also invariant under the transformation $(x,y)\mapsto(x+\pi,-y)$ and thus descends to an $\hoe$-valued connection $1$-form on the Klein bottle, $K$. 
Consequently, $K\times N$ also admits a tangent $2$-plane field of Cartan type transverse to the fibers $*\times N$.
\end{example}

\end{document}